\documentclass[preprint,12pt]{elsarticle}
\DeclareSymbolFont{cmbrightop}{OT1}{cmbr}{m}{n}
\DeclareMathSymbol{\sfPsi}{\mathalpha}{cmbrightop}{9}

\DeclareFontFamily{U}{nxlmi}{}
\DeclareFontSubstitution{U}{nxlmi}{m}{it}
\DeclareFontShape{U}{nxlmi}{m}{it}{
	<-6.3>    nxlmi05
	<6.3-8.6> nxlmi07
	<8.6->    nxlmi0
}{}
\DeclareFontShape{U}{nxlmi}{b}{it}{
	<-6.3>    nxlbmi05
	<6.3-8.6> nxlbmi07
	<8.6->    nxlbmi0
}{}
\renewcommand{\partial}{{\text{\usefont{U}{nxlmi}{m}{it}\symbol{64}}\mspace{1mu}}}

\usepackage{amssymb}

\usepackage{bm, bbm}
\usepackage{amsthm, amsmath,cases, upgreek,mathrsfs }
\usepackage{newtxtext, newtxmath}

\biboptions{numbers,compress}

\numberwithin{equation}{section}
\newtheorem{theorem}{Theorem}[section]
\newtheorem{lemma}[theorem]{Lemma}

\newtheorem{proposition}[theorem]{Proposition}
\newtheorem{remark}[theorem]{Remark}

\newcommand{\eps}{\varepsilon}

\newcommand{\dif}{\ensuremath{\,\mathrm{d}}}

\renewcommand{\div}{\mathop{\mathrm{div}}\nolimits}
\newcommand{\curl}{\mathop{\mathrm{curl}}\nolimits}

\journal{Journal of Functional Analysis}

\begin{document}

\begin{frontmatter}



\title{Global well-posedness of solutions for the equations modelling the motion of a rigid body in a bidimensional perfect fluid}


\author{Xiaoguang You}

\affiliation{organization={School of Mathematical Sciences, Jiangxi Science and Technology Normal University},
            city={Nanchang},
            postcode={330038}, 
            state={Jiangxi},
            country={Peoples R China}}

\begin{abstract}
This paper considers a system modelling the evolution of a rigid body immersed in a bidimensional incompressible perfect fluid. In the special case of a disk-shaped rigid body, it was shown by C. Rosier and L. Rosier (2009) that the system admits a unique global solution when the initial fluid velocity $u_0$ belongs to $H^s$ ($s \geqslant 3$) and its vorticity $\curl u_0$ lies in $L^p$ with $1 \leqslant p < 2$. By establishing a Beale–Kato–Majda type bound, we generalize the result by removing the constraint $\curl u_0 \in L^p$ and allowing the rigid body to be of arbitrary shape. Moreover, we obtain an explicit energy bound.
\end{abstract}


\end{frontmatter}


\section{Introduction}\label{Sec:introduction}
This paper investigates a system that models the motion of a rigid body within an incompressible perfect fluid. The fluid dynamics are governed by the bidimensional Euler equations, while the motion of the rigid body follows from the conservation of linear and angular momentum. More precisely,  let \(\mathcal{S} \subset \mathbb{R}^2\) be a smooth, simply connected, and bounded open set representing the region initially occupied by the rigid body, and set \(\mathcal{F} = \mathbb{R}^2 \setminus \overline{\mathcal{S}}\). We denote by \(\mathcal{S}(t)\) the region occupied by the rigid body and by \(\mathcal{F}(t)\) the region occupied by the fluid at time \(t \geqslant 0\). Let $n(x, t)$ be the unit normal vector field on $\partial\mathcal{S}(t)$, which directs toward the interior of $\mathcal{S}(t)$. The fluid is assumed to be homogeneous, with density \(\rho_f = 1\). We denote by \(u(x,t)\) and \(p(x,t)\) the velocity field and pressure within the fluid, respectively. For any \(t > 0\), the fluid motion is governed by the Euler equations:
\begin{subequations}\label{Equ:euler-1}
	\begin{align}
		\partial_t u + (u \cdot \nabla )u + \nabla p = 0, \quad x \in \mathcal{F}(t),\label{Equ:euler-1a}\\
	\div u = 0,  \quad  x \in \mathcal{F}(t),\label{Equ:euler-1b}
	\end{align}
\end{subequations}
and the evolution law of the rigid body is given by
\begin{subequations}\label{Equ:euler-2}
\begin{align}
	mh''(t) = \int_{\partial \mathcal{S}(t)} p\,n \dif \sigma,\label{Equ:euler-2a}\\
	Jr'(t) = \int_{\partial \mathcal{S}(t)} p\,(x - h(t))^\perp \cdot n \dif \sigma \label{Equ:euler-2b}.
\end{align}
\end{subequations}
Here, we denote \(x^\perp = (-x_2, x_1)\) for \(x = (x_1, x_2) \in \mathbb{R}^2\), and \(\mathrm{d} \sigma\) stands for the integration element along the boundary \(\partial \mathcal{S}(t)\). The quantities \(m\) and \(J\) are the mass and the inertia of the rigid body, respectively. Moreover, \(h(t)\) denotes the position of the body's center of mass, and \(r(t)\) represents the angular velocity of the body.

On the boundary, we assume
\begin{subequations}\label{Equ:euler-3}
	\begin{align}
		u \cdot n =  \big(h' + r(x - h)^\perp\big) \cdot n \text{ on } \partial\mathcal{S}(t),\label{Equ:euler-3a}\\
		\lim_{|x|\rightarrow \infty } u(x, t) = 0\label{Equ:euler-3b},
	\end{align}
\end{subequations}
and the initial conditions of the system are given by
\begin{subequations}\label{Equ:euler-4}
	\begin{align}
		u(x, 0) = u_0(x), \quad x \in \mathcal{F},\label{Equ:euler-4a}\\
		h(0) = 0\in \mathbb{R}^2,\quad \dot{h}(0) = \ell_0 \in \mathbb{R}^2,\quad r(0) = r_0 \in \mathbb{R},\label{Equ:euler-4b}
	\end{align}
\end{subequations}
where we have assumed the initial position of the body being at the origin. 

As is common in many fluid--body interaction problems, the main challenge in establishing the well-posedness of the system \eqref{Equ:euler-1}--\eqref{Equ:euler-4} arises from its nonlinearity, strong coupling, and the fact that the fluid domain is unknown and time-dependent. Early results on global existence relied on strong assumptions on the initial vorticity. Notably, Ortega et al.~\cite{ortega2005,ortega2007} established global existence of smooth solutions under the condition that the initial vorticity belongs to a weighted Sobolev space. Subsequently, for the special case where the rigid body is a unit disk, C.~Rosier and L.~Rosier~\cite{rosier2009}  proved global existence assuming only that the initial velocity lies in \(H^s\) (\(s \geqslant 3\)) and the initial vorticity belongs to \(L^p\) for some \(p \in [1,2)\). These vorticity restrictions originate from the Biot--Savart law, whose kernel exhibits opposite singular behaviour near the boundary and at infinity. In fact, even for a fixed rigid body, Kikuchi~\cite{kikuchi1983exterior} found it necessary to impose similar vorticity conditions to obtain global smooth solutions.

A natural question then arises: for a general domain \(\mathcal S\) (i.e., an arbitrarily shaped smooth rigid body), can the system \eqref{Equ:euler-1}--\eqref{Equ:euler-4} admit a global smooth solution without imposing any additional conditions on the vorticity? The present paper provides an affirmative answer to this question. The crucial point is our novel treatment of the Poisson equation in an exterior domain: through highly delicate calculations, we derive a Beale--Kato--Majda type bound. This estimate, analogous to the whole-space case, requires no additional assumptions on the vorticity.

Before stating the main result, we introduce some functional spaces. Let $\mathcal O$ be an open subset of $\mathbb{R}^2$, we denote by \(W^{s,p}({\mathcal{O}})\) the usual \(L^p\)-based Sobolev spaces of order \(s\), and simply write \(H^s(\mathcal{O})\) when \(p=2\). Let \(C_0^\infty(\mathcal{O})\) denote the space of smooth, compactly supported functions in \(\mathcal{O}\), and let \(H_0^s(\mathcal{O})\) be its closure with respect to the \(H^s\)-norm. The subspace of \(C_0^\infty(\mathcal{O})\) consisting of divergence‑free vector fields is denoted by \(C_{0,\sigma}^\infty(\mathcal{O})\); accordingly, \(L_\sigma^2(\mathcal{O}) := \overline{C_{0,\sigma}^\infty(\mathcal{O})}^{\,\|\cdot\|_{L^2}}\). We will also make use of the homogeneous Sobolev space, defined as
\begin{align*}
	\dot{H}^s(\mathcal{O}) := \Big\{ f \in L_{\mathrm{loc}}^2(\overline{\mathcal{O}}) \ \Big | \ \nabla f \in H^{s-1}(\mathcal{O}) \Big\},
\end{align*}
and \(\dot{H}^1_0(\mathcal{O})\) denotes the subspace of \(\dot{H}^1(\mathcal{O})\) consisting of functions which vanish on \(\partial\mathcal{O}\).

Let $0 < T \leqslant \infty$. Given a functional space $\mathscr Y$ consisting of functions of the spatial variable $x$, we denote by $C\big([0,T); \mathscr Y(\mathcal{F}(t))\big)$ the space of functions $f$ such that $f(\cdot,t) \in \mathscr Y(\mathcal{F}(t))$ for every $t$, and $f$ can be extended to a function in $C\big([0,T); \mathscr Y(\mathbb{R}^2)\big)$. We define the space \(\mathscr X_{s;T}\) by
\begin{equation}
	 \mathscr X_{s; T} = \left\{
	\begin{array}{c|c}
		(u, p, h, r) &
		\begin{array}{l}
			u \in C\big([0, T); H^s(\mathcal{F}(t))\big),\\
			p \in C\big([0, T); \dot{H}^{s}(\mathcal{F}(t))\big), \\
			h \in C^1\big([0, T); \mathbb{R}^2\big), r \in C\big([0, T); \mathbb{R}\big)
		\end{array}
	\end{array}
	\right\}.
\end{equation}

Suppose that \((u, p, h, r) \in \mathscr{X}_{s;T}\) is a solution to the system \eqref{Equ:euler-1}--\eqref{Equ:euler-4}. Then,  for \(0 \leqslant k \leqslant s\) and \(0 \leqslant t < T\), we define the energy \(E_k(t)\) of the system by
\begin{equation}
	E_k(t) = \|u(t)\|_{H^k(\mathcal{F}(t))}^2 + m|\dot{h}(t)|^2 + J|r(t)|^2,
\end{equation}
with \(E_k(0) = \|u_0\|_{H^k(\mathcal{F})}^2 + m|\ell_0|^2 + J|r_0|^2\).

We are now ready to state the main result.
\begin{theorem}\label{thm:well-posedness}
	Let $s \geqslant 3$ be an integer. Suppose that $u_0 \in H^s(\mathcal{F})$ satisfies
	\begin{equation} \label{Equ:boundary-initial}
		\begin{split}
			\div u_0 = 0 \text{ in } \mathcal{F}  \text{ and } 
			u_0 \cdot n = (\ell_0 + r_0\, x^\perp) \cdot n \text{ on } \partial\mathcal{S}.
		\end{split}
	\end{equation}
Then the system \eqref{Equ:euler-1}--\eqref{Equ:euler-4} admits a unique solution $(u, p, h, r) \in \mathscr X_{s;\infty}$.
	Moreover, there exist constants $K_i$ ($i=1,2,3$), independent of $s$, such that for all $t > 0$,
	\begin{align}\label{energy-inequality-thm}
		 E_s(t) \leqslant K_1 E_s(0) \exp \int_0^t \lambda(\tau) \dif \tau.
	\end{align}
	Here $\lambda(\tau)$ does not depend on $s$ and is given by
	\begin{equation}\label{Def:sigma}
		\lambda(\tau) = K_2\bigl(1+\sqrt{E_1(0)}\bigr)
		\bigl(1+\ln^+\sqrt{E_3(0)}\bigr)
		\exp\Bigl(K_3\bigl(1+\sqrt{E_1(0)}\bigr)\tau\Bigr),
	\end{equation}
	where $\ln^+ x = \max(0,\ln x)$.
\end{theorem}

\begin{remark}
	For the case where the rigid body is fixed in the fluid, the global existence of smooth solutions with initial velocity $u_0 \in H^s$ ($s \geqslant 3$) can be established similarly by using the Beale--Kato--Majda type bound obtained in Sect. 3. This extends the work of Kikuchi \cite{kikuchi1983exterior}, where the initial vorticity was assumed to lie in a weighted Sobolev space.
\end{remark}

The remainder of this paper is structured as follows. Sect. 2 presents some preliminary lemmas. In Sect. 3, we study the Poisson equation in exterior domains, where we establish a Beale--Kato--Majda type bound. A priori estimates are then derived in Sect. 4, and the proof of Theorem \ref{thm:well-posedness} is given in Sect. 5.

\section{Preliminaries}\label{Sec:pre}
	Let $\mathbb{D}$ denotes the unit disk centered at the origin. We first introduce a smooth biholomorphism $\mathcal{T}$ between $\mathcal{F}$ and $\mathbb{R}^2\setminus\overline{\mathbb{D}}$, which has the following properties. The proof may be found in \cite{iftimie2003two}. 
\begin{lemma}\label{lm:biholomorphism} There exists a smooth biholomorphism $\mathcal T: \mathcal{F} \rightarrow \mathbb{R}^2\setminus\overline{\mathbb{D}}$ that extends smoothly to $\partial \mathcal{S}$. Moreover, there exists a constant $K$ such that
	\begin{align}\label{Est:T-mapping}
		\|\nabla\mathcal{T}\|_{W^{1,\infty}(\mathcal{F})} + \|\nabla\mathcal{T}^{-1}\|_{W^{1,\infty}(\mathcal{F})} \leqslant K.
	\end{align}
\end{lemma}
Using the biholomorphism $\mathcal{T}$, we obtain an explicit formula for the Green's function of the Laplacian in $\mathcal{F}$:
\begin{align}
	G_\mathcal{F}(x, y) = \frac{1}{2\pi} \ln \frac{|\mathcal{T}(x) - \mathcal{T}(y)|}{|\mathcal{T}(x)- (\mathcal{T}(y))^*||\mathcal{T}(y)|},
\end{align}
where $\xi^*$ denotes
\[
\xi^* := \frac{\xi}{|\xi|^2}, \quad \forall \xi \in \mathbb{R}^2\setminus\mathbb{D}.
\]

Let us now introduce a Cauchy principal-value integral operator ${P}_{ij}$,  $i,j=1, 2$, defined by
\begin{equation}
	{P}_{ij}f(\eta) = \lim_{\eps \rightarrow 0}\int_{|\eta-\xi|>\eps} {\overline{\mathcal{K}}_{ij}}(\eta,\xi) f(\xi)d\xi,
\end{equation}
where the singular kernel $\overline{\mathcal{K}}_{ij}$ is given by
\begin{equation}\label{Def:p-k} 
	\overline{\mathcal{K}}_{ij}(\eta,\xi):=\frac{1}{2\pi}\partial_{\eta_i}\bigg(\frac{\eta_j-\xi_j}{|\eta-\xi|^2}\bigg).
\end{equation}

Lemma 4.6 of \cite{majda2002vorticity} shows that the operator ${P}_{ij}$ satisfies the following property.
\begin{lemma}\label{Lm:singular-integral-infty} Let $f$ be a scalar function supported in a bounded domain $\mathcal A$ (with measure denoted by $m_{\mathcal A}$). Then there exists a constant $K$, independent of $f$, such that for every $\epsilon > 0$ and $0 < \gamma < 1$, we have
	\begin{equation}\label{Est:pv-infty}
		\|{P}_{ij}f\|_{L^\infty(\mathbb{R}^2)} \leqslant K\big(\|f\|_{C^{\gamma}(\mathcal{A})}\eps^\gamma + \max\big(1,\ln ({\sqrt{m_{\mathcal A}}/{\eps}})\big)\|f\|_{L^\infty(\mathcal A)} \big).
	\end{equation}
\end{lemma}

We also need a potential theory estimate for solenoidal fields in $\mathbb{R}^2$, the proof of which can be found in Proposition 3.8 of \cite{majda2002vorticity}. 
\begin{lemma}\label{Lm:whole-nabla} Let $v\in H^3(\mathbb{R}^2)$ be a divergence-free vector field, and denote its curl by $\omega = \curl v$. Then
	\begin{equation}
		\|\nabla v\|_{L^\infty} \leqslant K(1 +  \ln^+\|v\|_{3})(1 + \|\omega\|_{L^\infty}).
	\end{equation}
\end{lemma}

We end this section with a result on the characterization of harmonic vector fields in bidimensional exterior domains, whose proof is given in \cite{hieber2020characterization}.
\begin{lemma}\label{lm:harmoinc} Suppose that $u \in L^2_\sigma(\mathcal{F})$ is irrotational. Then $u \equiv 0$ in $\mathcal{F}$.
\end{lemma}
	\begin{remark} Considering that a bidimensional exterior domain is not simply connected, the above conclusion is not immediately obvious. In fact, for $p >2$, there exists $u \neq 0$ in $L_\sigma^p(\mathcal{F})$ such that $\curl u \equiv 0$ in $\mathcal{F}$. For more details, we refer to \cite{hieber2020characterization}.
	\end{remark}

\section{The Poisson equation in an exterior domain}
Throughout this section, in order to simplify notation, all norms for Sobolev spaces are understood to be on the domain $\mathcal{F}$ unless otherwise specified. With loss of generality, we assume $\mathcal S \subset B(0, 1)$. For $R > 1$, let $\mathcal C_{R}$ denotes the annulus $\mathcal{F} \cap B(0, R)$, and define $ \mathcal C_{R_1, R_2}$ by
\begin{equation*}
	C_{R_1, R_2} = \big\{x\in \mathcal{F} \ \big| \  R_1 < |x| < R_2 \big\} , \ 0 < R_1 < R_2 < \infty.
\end{equation*} 

  We consider the Poisson problem
\begin{align}\label{Equation:Poisson}
	\begin{cases}
		\Delta \psi = \omega \hbox{ in } \mathcal{F},\\
		\psi = 0 \hbox{ on } \partial \mathcal{S}.
	\end{cases}
\end{align}

When $\omega \in L^2(\mathcal{F})$ has compact support, the following existence result holds.
\begin{proposition}[Existence]\label{Proposition:Poisson} Let $R > 0$ be fixed. Suppose that $\omega \in L^2(\mathcal{F})$  is supported in $\overline{\mathcal{C}_{{R}}}$. Then there exists a unique solution $\psi \in \dot{H}^1_0(\mathcal{F})$ to Eqs.~\eqref{Equation:Poisson}. Moreover, we have the following estimate
	\begin{align}\label{Estimate:Poisson-1}
		\|\nabla \psi\|_{1} \leqslant K_{R} \|\omega\|_{0},
	\end{align}
	where $K_{R}$ is a constant that depends on $R$. 
\end{proposition}

\begin{proof}
We first prove uniqueness. Assume that $\psi_1, \psi_2 \in \dot{H}^1_0(\mathcal{F})$ are two solutions of \eqref{Equation:Poisson}. Define $v := \nabla^\perp(\psi_1 - \psi_2)$. Then it can be checked that
\begin{equation}
	\operatorname{div} v = 0,\quad \operatorname{curl} v = 0,\quad v|_{\partial \mathcal{S}} \cdot n = 0,
\end{equation}
where $n$ denotes the unit normal vector on $\partial \mathcal{S}$. It follows from Lemma \ref{lm:harmoinc} that $v = 0$ in $\mathcal{F}$, which combining with the fact $\psi_1 = \psi_2 = 0$ on $\partial \mathcal{S}$ yields that $\psi_1 = \psi_2$ in $\mathcal{F}$.

Using Green's function $G_\mathcal{F}$, we can obtain an explicit formula  for $\psi$:
\begin{align}\label{Def:psi}
	\psi(x) = \int_\mathcal{F}  G_\mathcal{F}(x, y) \omega(y) \dif y.
\end{align}
It can be verified directly that  $\psi$ satisfies \eqref{Equation:Poisson}. Now, we divide the proof of bound \eqref{Estimate:Poisson-1} into two parts. 

\paragraph{\textbf{Part 1. Estimating $\nabla \psi$}}
Notice that, for $i=1,2$, $\partial_i\psi$ can be written as 	
	\begin{equation}\label{Equ:partial_psi}
		\partial_i \psi = \int_\mathcal{F} \big(\mathcal{K}(x,y) - \mathcal{K}^*(x,y)\big) \cdot \partial_i \mathcal{T}(x) \omega(y) \dif y,
	\end{equation}
	where
	\begin{equation}\label{Def:K_T}
		\mathcal{K}(x, y) := \frac{1}{2\pi}\frac{\mathcal{T}(x) - \mathcal{T}(y)}{|\mathcal{T}(x) - \mathcal{T}(y)|^2}, \quad \mathcal{K}^*(x, y) := \frac{1}{2\pi}\frac{\mathcal{T}(x) - (\mathcal{T}(y))^*}{|\mathcal{T}(x) - (\mathcal{T}(y))^*|^2}.
	\end{equation}
	
	Observing that $\mathcal{K}$ exhibits opposite behavior near the boundary and at infinity, it is necessary to treat these two regions separately.  More precisely, define $\widetilde{R}$ by
	\begin{align}
		\widetilde{R} := 2\Big(1 + \|\nabla\mathcal{T}\|_{L^\infty} + \|\nabla\mathcal{T}^{-1}\|_{L^\infty}\Big) R,
	\end{align}
	so that the integral of $|\partial_i\psi|^2$ over $\mathcal{F}$ satisfies:
	\begin{equation}\label{Est:partial_i_psi}
		\begin{split}
			&\int_{\mathcal{F}} |\partial_i\psi(x)|^2\dif x \leqslant 2\int_{\mathcal{C}_{\widetilde{R}}} \bigg| \int_{\mathcal{C}_R} \mathcal{K}(x,y)\cdot \partial_i \mathcal{T}(x) \omega(y) \dif y \bigg|^2 \dif x\\
			&\quad+ 2\int_{\mathcal{C}_{\widetilde{R}}} \bigg| \int_{\mathcal{C}_R} \mathcal{K}^*(x,y)\cdot \partial_i \mathcal{T}(x) \omega(y) \dif y \bigg|^2 \dif x \\
			&\quad +\int_{\mathcal{F}\setminus\mathcal{C}_{\widetilde{R}}} \bigg| \int_{\mathcal{C}_R} \big(\mathcal{K}(x,y) - \mathcal{K}^*(x,y)\big)\cdot \partial_i \mathcal{T}(x) \omega(y) \dif y\bigg|^2\dif x\\
			&=: I_1 + I_2 + I_3,
		\end{split}
	\end{equation}
	where we have used the assumption $\omega$ being supported in $\overline{\mathcal{C}_R}$.

In order to estimate $I_1$, we  first apply Lemma \ref{lm:biholomorphism} to deduce that
	\begin{equation}
		|\mathcal{K}(x,y)\cdot \partial_i \mathcal{T}(x)|  \leqslant K|x-y|^{-1}, \text{ for } x \in \mathcal{C}_{\widetilde{R}} \text{ and } y \in \mathcal C_{R},
	\end{equation}
	which combining with Young's convolution inequality yields that
	\begin{align}\label{Est:I-1}
		I_1 \leqslant K_{R}\|\omega\|_{0}^2,
	\end{align}
	where $K_{R}$ is a constant that depends on $R$. 
	
	The second term  $I_2$ can be handled similarly. Indeed, by using H\"older's inequality, we  obtain
	\begin{equation*}
		\begin{split}
		\bigg| \int_\mathcal{F} \mathcal{K}^*(x,y)\cdot \partial_i \mathcal{T}(x) \omega(y) \dif y \bigg|^2 &\leqslant K\int_{\mathcal{C}_R} \big|\mathcal{K}^*(x,y)\big| \,\big|\omega(y)\big|^2\dif y \int_{\mathcal{C}_R} \big|\mathcal{K}^*(x,y)\big|\dif y,
		\end{split}
	\end{equation*}
	which implies
	\begin{equation*}
		\begin{split}
			I_2 \leqslant K \|\omega\|_0^2 \times \sup_{y \in \mathcal C_R}\int_{\mathcal C_{\widetilde{R}}}\big|\mathcal{K}^*(x, y)\big| \dif x \times \sup_{x \in \mathcal C_{\widetilde{R}}}\int_{\mathcal C_R}\big|\mathcal{K}^*(x, y)\big| \dif y.
		\end{split}
	\end{equation*}	
	We can arrive at
	\begin{equation}
	I_2 \leqslant K_R\|\omega\|_0^2,
	\end{equation}
	once we prove
	\begin{equation}\label{Est:k}
		\begin{split}
			\sup_{y \in \mathcal C_R}\int_{\mathcal C_{\widetilde{R}}}\big|\mathcal{K}^*(x, y)\big| \dif x \leqslant K_R \text { and } \sup_{x \in \mathcal C_{\widetilde{R}}}\int_{\mathcal C_R}\big|\mathcal{K}^*(x, y)\big| \dif y \leqslant K_R.
		\end{split}
	\end{equation}
	Indeed, by making a change of variable
	\begin{equation*}
		\eta = \mathcal{T}(x) - (\mathcal{T}(y))^*,
	\end{equation*}
	we find that
	\[
		\mathcal{K}^*(x,y) \leqslant |\eta|^{-1} \text{ and } |\eta| \leqslant K({R} + 1) \text{ for } x \in \mathcal{C}_{\widetilde{R}} , y \in \mathcal C_R,
	\]
	and
	\[
\dif x = \det\Big(\nabla \mathcal T^{-1}\big(\mathcal T(x)\big)\Big)\dif \eta \text{ and } \dif y = \det\Big(\nabla \mathcal T^{-1}\big(\mathcal T(y)\big)\Big) |y|^4 \dif \eta,
\]	
	which yields \eqref{Est:k} from Lemma \ref{lm:biholomorphism}.
	
	To estimate $I_3$, we use the fact that for $a, b\in \mathbb{R}^2$ with $a \neq 0$ and $b \neq 0$, 
	\begin{equation*}
		\bigg|\frac{a}{|a|^2} - \frac{b}{|b|^2}\bigg| = \frac{|a-b|}{|a||b|},
	\end{equation*}
	from which it follows that for $x \in \mathcal{F}\setminus\mathcal{C}_{\widetilde{R}}$ and $y \in \mathcal{C}_R$,
	\[
		\big|\mathcal{K}(x,y) - \mathcal{K}^*(x,y)\big| = \frac{1}{2\pi}\,\frac{|\mathcal{T}(y) - (\mathcal{T}(y))^*|}{|\mathcal{T}(x)-\mathcal{T}(y)|\ |\mathcal{T}(x)-(\mathcal{T}(y))^*|} \leqslant KR^{-1}.
	\]
	Therefore,
	\begin{align}\label{Est:I-3}
		I_3 \leqslant K_{R}\|\omega\|_{0}^2.
	\end{align}
	
	With the estimates of $I_i$, i=1,2,3, it follows immediately
	\begin{equation}
		\|\partial_i\psi\|_0 \leqslant K_R \|\omega\|_{0}.
	\end{equation}
	
	\paragraph{\textbf{Part 2. Estimating $D^2 \psi$}} For $i, j=1,2$, the weak derivative $\partial_{ij}\psi$ is defined by
	\begin{equation}\label{Def:weak}
		\left<\partial_{ij}\psi,\varphi\right> := -\int_\mathcal{F} \partial_i\psi(x)\partial_j\varphi(x)\dif x, \ \forall \varphi \in C_0^\infty(\mathcal{F}).
	\end{equation}
	After integrating by parts, we find that
	\begin{equation}\label{Est:2-order-psi}
		\big|\left<\partial_{ij}\psi,\varphi\right>\big| = {J}_1 + {J}_2 + {J}_3+J_4,
	\end{equation}
	with 
	\begin{equation}\label{Def:j}\left\{
		\begin{aligned}
			&J_1 = -\int_{\mathcal{F}}\Big[\lim_{\eps \rightarrow 0}\int_{|x-y|=\eps} \mathcal{K}(x, y) \cdot \partial_i\mathcal{T}(x) \varphi(x) \frac{x_j}{|x|}\dif \sigma(x) \Big]\omega(y)\dif y,\\
			&J_2 = \int_{\mathcal{F}}\Big[\int_{\mathcal{F}} \partial_{x_j}\mathcal{K}(x, y) \cdot \partial_i\mathcal{T}(x) \varphi(x)\dif x \Big]\omega(y)\dif y,\\
			&J_3 = -\int_{\mathcal{F}}\Big[\int_{\mathcal{F}} \partial_{x_j}\mathcal{K}^*(x, y) \cdot \partial_i\mathcal{T}(x) \varphi(x)\dif x \Big]\omega(y)\dif y,\\
			&J_4 = \int_{\mathcal{F}}\int_{\mathcal{F}} \big(\mathcal{K}(x,y)-\mathcal{K}^*(x, y)\big) \cdot \partial_i\partial_j\mathcal{T}(x) \varphi(x)\omega(y)\dif x \dif y.
		\end{aligned}\right.
	\end{equation}
	
	The first term $J_1$ is easy to estimate. Indeed, Lemma \ref{lm:biholomorphism} implies
	\begin{equation}\label{Est:tmp-1}
		\big|\mathcal{K}(x,y)\big| \leqslant \frac{1}{2\pi}\big|\mathcal{T}(x)-\mathcal{T}(y)\big|^{-1}\leqslant K|x-y|^{-1},
	\end{equation}
	therefore, it follows from H\"older's inequality that
	\begin{equation}
		|J_1| \leqslant K\|\omega\|_{0}\|\varphi\|_{0}.
	\end{equation}
	
	To estimate $J_2$, we need to make the change of variables indicated by $\eta = \mathcal{T}(x), \xi =\mathcal{T}(y)$, then $\partial_{x_j}\mathcal{K}(x,y) \cdot \partial_i \mathcal T(x)$ can be rewritten by
	\begin{equation}
		\begin{split}
			\partial_{x_j}\mathcal{K}(x, y) \cdot \partial_i \mathcal T(x) = {\overline{\mathcal K}_{kl}}(\eta,\xi)   \partial_i \mathcal T_l\big(\mathcal T^{-1}(\eta)\big) \partial_j\mathcal{T}_k\big(\mathcal{T}^{-1}(\eta)\big),
		\end{split}
	\end{equation}
	where $\overline{\mathcal K}_{kl}$ is defined by \eqref{Def:p-k}. Therefore,
	\begin{equation}\label{Def:j_2}
		\begin{split}
			& J_2 = \int_{\mathbb{R}^2\setminus\mathbb{D}}\Big[\int_{\mathbb{R}^2\setminus\mathbb{D}} {\overline{\mathcal K}_{kl}}(\eta, \xi)\widetilde{\varphi}_{ijkl}(\eta) \dif \eta \Big] \widetilde{\omega}(\xi)\dif \xi .
		\end{split}
	\end{equation}
	with
	\begin{equation}\label{Def:tilde_w}
		\begin{split}
			\widetilde{\varphi}_{ijkl}(\eta) &= \partial_i\mathcal{T}_l\big(\mathcal{T}^{-1}(\eta)\big)\partial_j\mathcal{T}_k\big(\mathcal{T}^{-1}(\eta)\big) \det\big(\nabla \mathcal{T}^{-1}(\eta)\big)\varphi\big(\mathcal{T}^{-1}(\eta)\big),\\
			\widetilde{\omega}(\xi) &= \omega\big(\mathcal{T}^{-1}(\xi)\big)\det\big(\nabla\mathcal{T}^{-1}(\xi)\big).
		\end{split}
	\end{equation}
	Observing that ${\overline{\mathcal{K}}_{kl}}(\eta,\xi)$ is singular, we may deduce from Lemma \ref{lm:biholomorphism}, the Calder\'on-Zygmund theorem, and H\"older's inequality that
	\begin{equation}
		|J_2| \leqslant K\|\omega\|_{0}\|\varphi\|_{0}.
	\end{equation}
	
	$J_3$ can be estimated analogously. Indeed, let $\eta = \mathcal T(x)$ and $\xi=(\mathcal T(y))^*$, we get
	\begin{equation}
		\begin{split}
			& J_3 = -\int_{\mathbb{R}^2\setminus\mathbb{D}}\Big[\int_{\mathbb{R}^2\setminus\mathbb{D}} {\overline{\mathcal{K}}_{kl}}(\eta, \xi) \widetilde{\varphi}_{ijkl}(\eta)\dif \eta \Big]\widetilde{\omega}^*(\xi)\dif \xi,
		\end{split}
	\end{equation}
	with
	\[\widetilde{\omega}^* (\xi) = \omega\big(\mathcal{T}^{-1}(\xi^*)\big)\det\big(\nabla\mathcal{T}^{-1}(\xi^*)\big)|\xi|^{-4}.
	\]
Lemma \ref{lm:biholomorphism}, the Calder\'on-Zygmund theorem, and H\"older's inequality imply that
	\begin{equation}
		|J_3| \leqslant K\|\varphi\|_{0} \|\widetilde{\omega}^*\|_0.
	\end{equation}
	Furthermore, by using the fact that $\omega$ vanishes outside $\mathcal{C}_R$, we can obtain from Lemma \ref{lm:biholomorphism} that
	\[
	\|\widetilde{\omega}^*\|_0 \leqslant K\|\omega\|_0,
	\]
	which immediately yields
	\begin{equation}
		|J_3| \leqslant K_R\|\varphi\|_{0}\|\omega\|_{0}.
	\end{equation}
	
	To complete the estimate of $\partial_{i}\partial_j\psi$, it remains to address the last term $J_4$. Let us use Fubini's theorem to rewrite this term as
	\begin{equation}
		J_4 = \int_{\mathcal{F}}\Big[\int_{\mathcal{F}}\big(\mathcal{K}(x,y)-\mathcal{K}^*(x,y)\big)\omega(y)\dif y\Big]\cdot \partial_i\partial_j\mathcal{T}(x)\varphi(x)\dif x.
	\end{equation}
	
	Since  $D^2\mathcal{T}$ is bounded, we may follow strictly analogous calculations for $\|\nabla \psi\|_0$ to obtain
	\begin{equation}
		|J_4| \leqslant K_R\|\omega\|_{0}\|\varphi\|_{0}.
	\end{equation}

	Finally, by substituting the estimates of $J_{i}$, i=1...4, into \eqref{Est:2-order-psi}, it follows
	\begin{align}
		\big|\left<\partial_i\partial_j\psi, \varphi\right>\big| \leqslant K_{R}\|\varphi\|_{0}\|\omega\|_{0}.
	\end{align}
	Note that $\varphi \in C_0^\infty(\mathcal{F})$ is arbitrary, we conclude that the inequality \eqref{Estimate:Poisson-1} holds, and this completes the proof of Proposition \ref{Proposition:Poisson}.
\end{proof}

When $\omega$ possesses high regularity, we can derive further estimates for $\psi$, which are given as follows.

\begin{proposition}[Regularity]\label{Proposition:Poisson-2} Let $s\geqslant 0$ be an integer. Suppose that $\omega$ belongs to $H^s(\mathcal{F})$ and is compactly supported. Then we have 
	\begin{align}\label{Estimate:Poisson-2}
		\|D^{s+2}\psi\|_{0} \leqslant K(\|\omega\|_{s} + \|\nabla \psi\|_{0}),
	\end{align}
	where the constant $K$ is independent of the support of $\omega$.
\end{proposition}

\begin{proof}
	We split the proof into three steps. Choose radii $R_1<R_2<R_3$ such that $\mathcal{S} \subset\subset B(0,R_1)$. We first consider the Poisson equation in the whole plane to obtain a uniform bound of $D^{s+2}\psi$ on $\mathcal F \setminus\mathcal{C}_{R_2}$. Next, we study the Poisson equation on $\mathcal{C}_{R_3}$ to establish a further bound on $\mathcal{C}_{R_3}$. Finally, we collect these estimates together to derive the inequality \eqref{Estimate:Poisson-2}.
	
	\paragraph{\textbf{Step 1. Estimates on $\mathbb{R}^2$}} Let $\varphi \in C^\infty(\mathbb{R}^2)$ be a smooth function which is non-negative, vanishes on $B(0,R_1)$, and equals $1$ outside $B(0,R_2)$. Then the function $\widetilde{\psi}:=\varphi\psi$ satisfies the Poisson equation
	\begin{equation}
		\Delta \widetilde{\psi}(x) = \widetilde{\omega}(x), \ \forall x \in \mathbb{R}^2,
	\end{equation}
	with
	\[\widetilde{\omega} :=\varphi \omega + 2 \nabla \varphi \cdot \nabla \psi + \Delta \varphi \psi.
	\]
	It follows that
	\begin{equation}\label{Estimate:ds-phi-whole-1}
		\begin{split}
			\|D^{s+2} \widetilde{\psi}\|_{L^2(\mathbb{R}^2)} &\leqslant  K \|D^s\widetilde{\omega}\|_{L^2(\mathbb{R}^2)}.
		\end{split}
	\end{equation}
	Furthermore, since $\psi \equiv 0$ on $\partial \mathcal{S}$ and $\mathcal{C}_{R_2}$ is a bounded domain, Poincar\'e's inequality yields that
	\begin{equation}\label{Estimate:ds-phi-whole-t}
		\|\psi\|_{L^2(\mathcal{C}_{R_2})} \leqslant K\|\nabla \psi\|_{L^2(\mathcal{C}_{R_2})}.
	\end{equation}
	Therefore,
	\begin{equation}\label{Estimate:ds-phi-whole-3}
		\|D^{s+2}\psi\|_{L^2(\mathcal F\setminus\mathcal{C}_{R_2})}\leqslant K \big(\|\omega\|_{s} +\|\nabla \psi\|_{s}\big).
	\end{equation}

	\paragraph{\textbf{Step 2. Estimates on $\mathcal{C}_{R_3}$}}  Let $\widetilde \varphi \in C^\infty(\mathbb{R}^2)$ be a smooth function which is non-negative, vanishes on $B(0,R_2)$, and equals $1$ outside $B(0,R_3)$. Denote $g := \widetilde\varphi\psi$ and consider the following Poisson equation:
	\begin{align}
		\begin{cases}
			\Delta \psi = \omega \hbox{ in } \mathcal{C}_{R_3},\\
			\psi = g \hbox{ on } \partial\mathcal{C}_{R_3}.
		\end{cases}
	\end{align}
	We infer from Theorem 8.13 in \cite{gilbarg2001elliptic} that
	\begin{align}\label{Est:bounded-tmp-1}
		\|D^{s+2}\psi\|_{L^2(\mathcal{C}_{R_3})} \leqslant K\big(\|\omega\|_{H^s(\mathcal{C}_{R_3})} + \|\psi\|_{L^2(\mathcal{C}_{R_3})}  +\|g\|_{H^{s + 2}(\mathcal{C}_{R_3})}\big).
	\end{align} 
	
	Furthermore, using the fact $\psi$ and $\varphi$ vanishes on $\partial \mathcal{S}$ and $\mathcal C_{R_2}$, respectively, we may apply the Poincar\'e inequality and the Gagliardo–Nirenberg interpolation inequality to obtain
	\begin{equation}\label{Est:bounded-tmp-3}
		\|\psi\|_{L^2(\mathcal{C}_{R_3})}+ \|g\|_{H^{s+2}(\mathcal{C}_{R_3})}  \leqslant K\Big(\|D^{s+2}\psi\|_{L^2(\mathcal F\setminus \mathcal{C}_{R_2})}+\|\nabla\psi\|_{L^2(\mathcal{C}_{R_3})} \Big),
	\end{equation}
	from which we conclude that
	\begin{equation}\label{Estimate:ds-phi-whole-5}
		\begin{split}
			&\|D^{s+2}\psi\|_{L^2(\mathcal{C}_{R_3})} \leqslant K\Big(\|\omega\|_{s} +  \|D^{s+2}\psi\|_{L^2(\mathcal F \setminus \mathcal{C}_{R_2})}+  \|\nabla \psi\|_0\Big).
		\end{split}
	\end{equation}
	
	\paragraph{\textbf{Step 3. Merge the estimates}}
	Collecting \eqref{Estimate:ds-phi-whole-3} and \eqref{Estimate:ds-phi-whole-5}, we obtain immediately that
	\begin{align}
		\|D^{s+2}\psi\|_{0}	\leqslant K(\|\omega\|_{s} + \|\nabla \psi\|_{s}),
	\end{align}
	which yields \eqref{Estimate:Poisson-2} by an induction argument, and this ends the proof of Proposition \ref{Proposition:Poisson-2}.
\end{proof}

The next proposition provides an $L^\infty$ bound for $D^2 \psi$.
\begin{proposition}\label{Pr:infty-exterior}Suppose that $\omega \in H^2(\mathcal{F})$ is supported in $\overline{\mathcal C_{R}}$ with some $R >0$. Then, we have
	\begin{equation}\label{Est:nabla-compact}
		\|D^2\psi\|_{L^\infty} \leqslant K_R\big(1+\max(1, \ln \|\omega\|_2)\|\omega\|_{L^\infty}\big),
	\end{equation}
	where $K_R$ is a constant depends on $R$.
\end{proposition}

\begin{proof}
	Let $\varphi \in C_0^\infty(\mathcal{F})$ be arbitrary, recall from \eqref{Est:2-order-psi} that
		\begin{equation}
		\big|\left<\partial_{ij}\psi,\varphi\right>\big| = \sum_{i=1}^4 J_i,
	\end{equation}
	where $J_i$ is defined in \eqref{Def:j}.

	By using \eqref{Est:tmp-1} and H\"older's inequality, we deduce that
	\begin{equation}
		|J_1| \leqslant K\|\omega\|_{L^\infty}\|\varphi\|_{L^1}.
	\end{equation}
	
	To estimate $J_2$, we first rewrite \eqref{Def:j_2} as
		\begin{equation}
		\begin{split}
			& J_2 = \int_{\mathbb{R}^2\setminus\mathbb{D}}\Big[\int_{\mathbb{R}^2\setminus\mathbb{D}} {\overline{\mathcal{K}}_{kl}}(\eta, \xi)\widetilde{\omega}(\xi) \dif \xi \Big] \widetilde{\varphi}_{ijkl}(\eta)\dif \eta,
		\end{split}
	\end{equation}
	where $\widetilde{\omega}$ and $\widetilde{\varphi}$ are given in \eqref{Def:tilde_w}. Noting that $\omega$ is supported in $\overline{\mathcal C_R}$, hence there exists a bounded set $\mathcal A_R$ in which $\widetilde{\omega}$ is supported.  Then, we apply Lemma \ref{Lm:singular-integral-infty} and H\"older's inequality to deduce that
	\begin{equation}
		|J_2| \leqslant K\Big(\|\widetilde{\omega}\|_{C^{\gamma}(\mathcal A_R)}\eps^\gamma + \max\big(1,\ln ({\sqrt{m_{\mathcal A_R}}/{\eps}})\big)\|\widetilde \omega\|_{L^\infty(\mathcal A_R)} \Big) \|\widetilde\varphi\|_{L^1(\mathcal{F})},
	\end{equation}
	and we infer from Lemma \ref{lm:biholomorphism} that
	\begin{equation}\left\{
		\begin{split}
			\|\widetilde{\omega}\|_{L^\infty(\mathcal A_R)} &\leqslant K\|\omega\|_{L^\infty(\mathcal{C}_R)},\\
			 \|\widetilde{\omega}\|_{C^\gamma(\mathcal A_R)} &\leqslant K\|\omega\|_{C^\gamma(\mathcal C_R)},\\
			\|\widetilde\varphi\|_{L^1(\mathcal{F})} &\leqslant K\|\varphi\|_{L^1(\mathcal{F})}.
		\end{split}\right.
	\end{equation}
Therefore, by setting $\eps = 1$ if $\|{\omega}\|_2 \leqslant 1$, and $\eps = \|{\omega}\|_2^{-\frac{1}{\gamma}}$ otherwise, we get
\begin{equation}
	|J_2| \leqslant K_R\big(1+\max(1, \ln \|\omega\|_2)\|\omega\|_{L^\infty}\big)\|\varphi\|_{L^1}.
\end{equation}

Let us now estimate $J_3$, which can be rewritten as
\begin{equation}
	\begin{split}
		& J_3 = -\int_{\mathbb{R}^2\setminus\mathbb{D}}\Big[\int_{\mathbb{R}^2\setminus\mathbb{D}} {\overline{\mathcal{ K}}_{kl}}(\eta, \xi)\widetilde{\omega}^*(\xi)\dif \xi \Big]\widetilde{\varphi}_{ijkl}(\eta)\dif \eta ,
	\end{split}
\end{equation}
Observing that $\omega$ is supported in $\mathcal C_R$, there must exists $0< \widetilde{R} < 1$ such that $\widetilde{\omega}^*$ is supported in $\overline{\mathcal C_{\widetilde{R}, 1}}$. Then, using again Lemma \ref{Lm:singular-integral-infty} and H\"older's inequality, we arrive at
	\begin{equation}
	|J_3| \leqslant K\Big(\|\widetilde{\omega}^*\|_{C^{\gamma}(\mathcal C_{\widetilde{R}, 1})}\eps^\gamma + \max\big(1,\ln ({\sqrt{m_{\mathcal C_{{\widetilde{R},1}}}}/{\eps}})\big)\|\widetilde \omega^*\|_{L^\infty(\mathcal C_{\widetilde R, 1})} \Big) \|\widetilde\varphi\|_{L^1(\mathcal{F})}.
\end{equation}
Furthermore, by using Lemma \ref{lm:biholomorphism}, it can be checked that
\begin{equation}
	\begin{split}
		&\|\widetilde{\omega}^*\|_{L^\infty(\mathcal C_{\widetilde{R}, 1})} \leqslant K\|\omega\|_{L^\infty(\mathcal{C}_R)}\text{ and } \|\widetilde{\omega}^*\|_{C^\gamma(\mathcal C_{\widetilde{R}, 1})} \leqslant K\|\omega\|_{C^\gamma(\mathcal C_R)}.
	\end{split}
\end{equation}
Then, we take $\eps = 1$ if $\| \omega\|_2 \leqslant 1$ and $\|  \omega\|_2^{-\frac{1}{\gamma}}$ otherwise. It follows that
\begin{equation}
|J_3| \leqslant K_R\big(1+\max(1, \ln \|\omega\|_2)\|\omega\|_{L^\infty}\big)\|\varphi\|_{L^1}.
\end{equation}

It remains to estimate $J_4$, which reads as
	\begin{equation}\label{Def:j4}
		\begin{split}
	J_4 =& \int_{\mathcal{F}}\Big[\int_{\mathcal{F}}\big(\mathcal{K}(x,y) - \mathcal K^*(x, y)\big)\omega(y)\dif y\Big]\cdot \partial_i\partial_j\mathcal{T}(x)\varphi(x)\dif x.
	\end{split}
\end{equation}
Noting that $\mathcal K(x, y) \approx |x-y|^{-1}$, which exhibits completely different asymptotic behavior near zero and at infinity, we thus need to partition the integral  region. Indeed, on the one hand, by using H\"older's inequality, we infer that
\begin{equation}
	\Big|\int_{|x-y|<1}\mathcal{K}(x,y)\omega(y)\dif y\Big| \leqslant K\|\omega\|_{L^\infty}.
\end{equation}
On the other hand, by combining with the fact that $\omega$ is supported in $\overline{\mathcal{C}_R}$, we obtain
\begin{equation}
	\Big|\int_{|x -y|>1}\mathcal{K}(x,y)\omega(y)\dif y\Big| \leqslant K\|\omega\|_{L^1} \leqslant K_R \|\omega\|_{L^\infty}.
\end{equation}
Therefore, we apply H\"older's inequality to conclude that
\begin{equation}\label{Est:k-1}
	\int_{\mathcal{F}}\Big[\int_{\mathcal{F} }\mathcal{K}(x,y)\omega(y)\dif y\Big]\cdot \partial_i\partial_j\mathcal{T}(x)\varphi(x)\dif x \leqslant K_R\|\omega\|_{L^\infty}\|\varphi\|_{L^1}.
\end{equation}

The part involving the kernel $\mathcal K^*$ in integral \eqref{Def:j4} can be handled analogously. Indeed, for  $|\mathcal T(x) > 1$ and $|\mathcal T(y)| > 2$, we have $\mathcal K^*(x, y) < 2$. Hence,
\begin{equation}\label{Est:w-1}
	\Big|\int_{|T(y)|>2}\mathcal{K}^*(x,y)\omega(y)\dif y\Big| \leqslant K\|\omega\|_{L^1} \leqslant K_R \|\omega\|_{L^\infty}.
\end{equation}
To consider the integral over the region where $|\mathcal T(y)| < 2$, we need make the change of variables $\eta = \big(T(y)\big)^*$. Then, we see that
\begin{equation*}
	\Big|\int_{|T(y)|<2\cap \mathcal{F}}\mathcal{K}^*(x,y)\omega(y)\dif y\Big| \leqslant K\int_{\frac{1}{2} <|\eta| < 1} |\mathcal T(x) - \eta|^{-1} |\omega(\mathcal T^{-1}(\eta^*))|\dif \eta,
\end{equation*}
which implies
\begin{equation}\label{Est:w-2}
	\Big|\int_{|T(y)|<2\cap \mathcal{F}}\mathcal{K}^*(x,y)\omega(y)\dif y\Big| \leqslant K\|\omega\|_{L^\infty}.
\end{equation}
Therefore,
\begin{equation}\label{Est:k-2}
		\int_{\mathcal{F}}\Big[\int_{\mathcal{F} }\mathcal{K}^*(x,y)\omega(y)\dif y\Big]\cdot \partial_i\partial_j\mathcal{T}(x)\varphi(x)\dif x \leqslant K_R\|\omega\|_{L^\infty}\|\varphi\|_{L^1}.
\end{equation}
Then, we infer from \eqref{Est:k-1} and \eqref{Est:k-2} that
\begin{equation}
	|J_4| \leqslant K_R\|\omega\|_{L^\infty}\|\varphi\|_{L^1}.
\end{equation}

Combining the bounds for $J_i$, $i=1,\dots,4$, in which $\varphi$ is arbitrary, immediately yields \eqref{Est:nabla-compact}. This completes the proof of Proposition \ref{Pr:infty-exterior}.
\end{proof}

The above results yield several important estimates for divergence-free vector fields. More precisely, let \(s > 0\) be an integer, and suppose that \(u \in L^2_\sigma(\mathcal{F}) \cap H^s(\mathcal{F})\). Define \(\omega = \curl u\). By Lemma 2.3 of \cite{You2024}, there exists a scalar function \(\psi \in \dot{H}_0^1\) such that \(u = \nabla^\perp \psi\). Let $\varphi$ be an arbitrary smooth function in $\mathbb{R}^2$ such that $\varphi \equiv 1$ in $B(0, 1)$, $0 \leqslant \varphi \leqslant 1$, and $\varphi$ vanishes outside $B(0, 2)$. For $\eps > 0$, set $\varphi^\eps(x) = \varphi(\eps x)$, $u^\eps = \nabla^\perp (\varphi^\eps\psi)$, and $\omega^\eps = \curl u^\eps$. Since \(\omega^\eps\) is compactly supported, there exists a solution \(\psi^\eps \in \dot{H}_0^1(\mathcal F)\) to Eqs.~\eqref{Equation:Poisson} with \(\omega\) replaced by \(\omega^\eps\). It then follows from Proposition \ref{Proposition:Poisson-2} that
\begin{equation}
	\|D^{s} u^\eps\|_0 \leqslant K\big(\|\omega^\eps\|_{s-1} + \|u^\eps\|_0\big).
\end{equation}
Taking \(\eps \rightarrow 0\), we obtain
\begin{equation}\label{Estimate:u-s}
	\|D^{s} u\|_0 \leqslant K\big(\|\omega\|_{s-1} + \|u\|_0\big).
\end{equation}
Moreover, we have the following Beale--Kato--Majda type bound for the field $u$.
\begin{proposition}\label{Pro:Poisson-3} Suppose that $u\in L^2_\sigma(\mathcal{F}) \cap H^s(\mathcal{F})$. Denote $\omega = \curl u$. Then, there exists a constant $K$ such that
	\begin{equation}\label{Estimate:u-infty}
		\begin{aligned}
			\|\nabla u\|_{L^\infty}  \leqslant K(1+\ln^+\|u\|_{3})(1+\|u\|_{0}+\|\omega\|_{L^\infty}).
		\end{aligned}
	\end{equation}
	where $\ln^+x$ denotes $\ln x$ for $x > 1$ and 0 otherwise.
\end{proposition}

\begin{proof}
	Without loss of generality, we assume that $\mathcal{S} \subset\!\subset B(0, 1)$, and let $\varphi$ be an arbitrary smooth function in $\mathbb{R}^2$ such that $\varphi \equiv 1$ in $B(0, 1)$, $0 \leqslant \varphi \leqslant 1$, and $\varphi$ vanishes outside $B(0, 2)$. Then, we split $\psi$ into two pieces, i.e.,
	\begin{equation}\label{Def:Psi-divided}
		\psi = \varphi\psi + (1-\varphi)\psi=: \psi_{b} + \psi_h.
	\end{equation} 
	On the one hand, it can be checked that $\psi_b$ satisfies the following Poisson equation in $\mathcal{F}$:
	\begin{equation}
		\begin{split}
			&\Delta \psi_b = \omega_b, \quad \psi|_{\partial \mathcal{S}} = 0, 
		\end{split}
	\end{equation}
	where
	\begin{equation}
		\begin{split}
			\omega_b  &= \varphi\omega  + 2 \nabla \varphi \cdot \nabla \psi + \Delta \varphi\psi.
		\end{split}
	\end{equation}
	Noting that $\omega_b$ is supported in $\overline{B(0, 2)}$,  we infer from Proposition \ref{Pr:infty-exterior} that
	\begin{equation}\label{Est:nabla-compact-1}
		\|D^2\psi_b\|_{L^\infty} \leqslant K\big(1+\max(1, \ln \|\omega_b\|_2)\|\omega_b\|_{L^\infty}\big).
	\end{equation}
	It is notable that the constant $K$ in the above inequality does not depend on the support of $\omega$.
	
	On the other hand, by using the fact $(1-\varphi)$ being zero in $B(0, 1)$, we infer that $\psi$ satisfies the Poisson equation in the full-plane, i.e.,
	\begin{equation}
		\Delta \psi_h = \omega_h,
	\end{equation}
	with
	\[
	\omega_h =   (1-\varphi)\omega  - 2 \nabla \varphi \cdot \nabla \psi - \Delta \varphi\psi.
	\]
	It follows from Lemma \ref{Lm:whole-nabla} that
	\begin{equation}\label{Est:h}
		\|D^2\psi_h\|_{L^\infty}  \leqslant K(1+\ln^+\|\nabla\psi_h\|_{3})(1+\|\omega_h\|_{L^\infty}).
	\end{equation}
	
	Collecting \eqref{Est:nabla-compact-1} and \eqref{Est:h}, we deduce that it is sufficient to verify
	\begin{equation}\label{Est:t1}
		\|\omega_b\|_{L^\infty} + \|\omega_h\|_{L^\infty} \leqslant K(\|u\|_0 + \|\omega\|_{L^\infty}).
	\end{equation}
	
	Let us first treat with $\|\Delta \varphi \psi\|_{L^\infty}$. Indeed, by using the fact $\Delta \varphi$ being supported in $\overline{\mathcal{C}_{1,2}}$, we infer from  Theorem 9.11 in \cite{gilbarg2001elliptic} that
	\begin{equation}
		\|\Delta \varphi \psi\|_{L^\infty} \leqslant  K\|\psi\|_{H^2(\mathcal{C}_{1,2})} \leqslant K(\|\psi\|_{L^2(\mathcal C_{3})} + \|\omega\|_{L^2(\mathcal C_{3})}).
	\end{equation}
	Furthermore, since $\psi$ vanishes on $\partial \mathcal{S}$, we apply Poincar\'e's inequality to obtain
	\begin{equation}
		\|\psi\|_{L^2(\mathcal C_{3})} \leqslant \|u\|_{L^2(\mathcal C_{3})}.
	\end{equation}
	Therefore,
	\begin{equation}\label{Est-delta-varphi-psi}
		\|\Delta \varphi \psi\|_{L^\infty} \leqslant  K(\|u\|_{0} + \|\omega\|_{L^\infty}).
	\end{equation}

	The term $\|\nabla\varphi\cdot\nabla\psi\|_{L^\infty}$ can be estimated similarly. Indeed, using Morrey’s inequality, we know that
	\begin{equation}
		\|\nabla\varphi\cdot \nabla \psi\|_{L^\infty} \leqslant \|\nabla \psi\|_{W^{1,4}(\mathcal{C}_{1,2})}.
	\end{equation}
	Then, by applying Theorem 9.11 in \cite{gilbarg2001elliptic} again, we obtain
	\begin{equation}\label{Est-delta-varphi-psi-1}
		\|\nabla\psi\|_{W^{1,4}(\mathcal{C}_{1,2})} \leqslant K(\|\psi\|_{L^4(\mathcal{C}_{3})} +\|\omega\|_{L^4(\mathcal{C}_{3})}).
	\end{equation}
	Furthermore,  Ladyzhenskaya's inequality and Poincar\'e's inequality tell
	\begin{equation}\label{Est-delta-varphi-psi-2}
		\|\psi\|_{L^4(\mathcal{C}_{3})} \leqslant K\|\psi\|_{H^1(\mathcal{C}_{3})} \leqslant K\|u\|_{L^2(\mathcal{C}_{3})},
	\end{equation}
	which implies
	\begin{equation}\label{Est:-nabla-varphi-nabla-psi}
		\|\nabla\varphi\cdot \nabla \psi\|_{L^\infty} \leqslant K(\|u\|_{0} + \|\omega\|_{L^\infty}).
	\end{equation} 
	
	Collecting \eqref{Est-delta-varphi-psi} and \eqref{Est:-nabla-varphi-nabla-psi}, we obtain \eqref{Est:t1} immediately. This completes the proof of Proposition \ref{Pro:Poisson-3}.
	
\end{proof}

\section{A priori estimates}

Unless otherwise stated, the $H^s$ and $L^\infty$ norms appearing in this section are defined with respect to the domain $\mathcal{F}(t)$. We first show that the kinetic energy $E_0(t)$ is conserved.
\begin{proposition}\label{Pr:uniform-low}
	Let $s \geqslant 3$ be fixed. Suppose that $(u, p, h, r) \in X_{s;\infty}$ is a solution to \eqref{Equ:euler-1}--\eqref{Equ:euler-4}. Then for all $ t > 0$, we have
	\begin{align}\label{energy-inequality-low}
	E_0(t) = E_0(0).
	\end{align}
\end{proposition}
\begin{proof}
	For $t \in (0, \infty)$, by taking the inner product of \eqref{Equ:euler-1a} with $u$, we obtain
	\begin{equation}\label{Equ:low-id1}
		\int_{\mathcal{F}(t)} \partial_t u \cdot u \dif x + \int_{\mathcal{F}(t)} [u\cdot \nabla u] \cdot u \dif x + \int_{\mathcal{F}(t)} \nabla p \cdot u \dif x = 0.
	\end{equation}
	On the one hand, the Reynolds transport theorem combined with the fact that $\div u = 0$ in $\mathcal{F}(t)$ implies
	\begin{equation}\label{Equ:low-id2}\int_{\mathcal{F}(t)} \partial_t u \cdot u \dif x + \int_{\mathcal{F}(t)} [u\cdot \nabla u] \cdot u \dif x = \frac{1}{2}\frac{\dif}{\dif t} \int_{\mathcal{F}(t)}|u|^2 \dif x.
	\end{equation}
	On the other hand, by using \eqref{Equ:euler-2a}, \eqref{Equ:euler-2b} and \eqref{Equ:euler-3a} we obtain
	\begin{equation}\label{Equ:low-id3}
		\begin{split}
			\int_{\mathcal{F}(t)} \nabla p \cdot u \dif x &= \int_{\partial\mathcal{S}(t)} p u \cdot n = \frac{m}{2}\frac{\dif}{\dif t}|h'(t)|^2 + \frac{J}{2}\frac{\dif}{\dif t}|r(t)|^2.
		\end{split}
	\end{equation}
	Therefore, by substituting \eqref{Equ:low-id2} and \eqref{Equ:low-id3} into \eqref{Equ:low-id1}, we obtain 
	\begin{equation}
		\frac{\dif}{\dif t}\Big[\int_{\mathcal{F}(t)} |u|^2 \dif x + m|h'|^2 + J|r|^2\Big] = 0,
	\end{equation}
	which yields immediately \eqref{energy-inequality-low} by integrating over $(0, t)$.
\end{proof}

In order to obtain high-order derivatives of $u$, we need to establish a bound of $\|\nabla u\|_{L^\infty}$. It is notable that the velocity field $u$ does not satisfy the slip condition on the boundary $\partial\mathcal{S}(t)$, thus Proposition \ref{Pro:Poisson-3} cannot be applied directly. To address this issue,  we construct manually a boundary corrector $\Lambda$ that is smooth and satisfies $\Lambda\cdot n = u \cdot n$ on $\mathcal{S}(t)$. To this end, we begin by introducing two auxiliary functions:
\begin{equation}
	\begin{split}
		&v(x, t) = {h}'(t) + r(t)(x - h(t))^\perp, \quad \forall x \in \mathbb{R}^2, \ t\in [0, \infty),\\
		&U(x, t) = -{h}'(t)^\perp \cdot  (x - h(t)) + r(t) \frac{|x-h(t)|^2}{2}, \quad \forall x \in \mathbb{R}^2,\ t\in [0, \infty).
	\end{split}
\end{equation}
Without loss of generality, we assume $\mathcal{S} \subset\!\subset B(0, 1)$. Let $\varphi$ be a smooth function on $\mathbb{R}^2$ satisfying $\varphi(x) \equiv 0$ for $0 \leqslant |x| \leqslant 1$, $\varphi \geqslant 0$, and $\varphi \equiv 1$ for $|x| \geqslant 2$. We denote $\varphi_h(x) = \varphi(x - h(t))$.
Then, for all $x \in \mathbb{R}^2$ and  $t\in [0, \infty)$, we define $\Lambda(x, t)$ by
\begin{equation}\label{Def:Lambda}
	\Lambda(x, t) = (\frac{\partial \varphi_h}{\partial x_2} U + (1-\varphi_h)v_1, -\frac{\partial \varphi_h}{\partial x_1} U + (1-\varphi_h)v_2).
\end{equation}
\begin{lemma}\label{Lm:boundary-corrector} Suppose $(u, p, h, r) \in X_{s;\infty}$ is a solution to \eqref{Equ:euler-1}--\eqref{Equ:euler-4}. Then, for all $t \in [0, \infty)$,  we have
	\begin{equation}\label{Pr:property}\left\{
		\begin{aligned} &\Lambda = 0 \text{ outside } B(h(t),2) \text{ and } \Lambda \cdot n = u \cdot n \text{ on } \partial\mathcal{S}(t),\\
		 &\div \Lambda = 0 \text{ in } \mathbb{R}^2.\\
		 \end{aligned}\right.
	\end{equation}
\end{lemma}

\begin{proof} Noting that $\varphi_h$ is supported in the $\overline{B(h(t), 2)}$, we thus have that $\Lambda$ vanishes outside $B(h(t), 2)$. Furthermore,  for those $x \in \partial\mathcal{S}(t)$, it can be checked that $\varphi_h(x) = 0$ and $\nabla \varphi_h(x) = 0$, which implies $\Lambda \cdot n = u \cdot n$ on $\partial\mathcal{S}(t)$. The  solenoidality of $\Lambda$ follows straightforwardly from the definition. This completes the proof of Lemma \ref{Lm:boundary-corrector}.
\end{proof}

With the above lemma, we are ready to establish a $L^\infty$ bound for $\nabla u$.

\begin{proposition}\label{Pr:nabla-u-infty} Suppose $(u, p, h, r) \in \mathscr X_{s; \infty}$ is a solution of \eqref{Equ:euler-1}--\eqref{Equ:euler-4}, and set $\omega = \curl u$. Then we have
	\begin{equation}\label{Est:infty-u}
		\begin{split}
			\|\nabla& u\|_{L^\infty}  \leqslant K\big(1+\ln^+\sqrt{E_3(t)}\big)\big(1+\sqrt{E_0(t)}+\|\omega_0\|_{L^\infty}\big).
		\end{split}
	\end{equation}
\end{proposition}
\begin{proof}
	We first set
	\begin{equation}
		\theta(t) = \int_0^t r(\tau)\dif \tau, \quad  Q(t) = \begin{pmatrix}
			\cos\theta(t) & -\sin \theta(t)\\
			\sin \theta(t) & \cos \theta(t)
		\end{pmatrix}.
	\end{equation}
	
	Then, we make a change of variables $x = Q(t)y + h(t)$ and define 
	\[v(y) = u\big( Q(t)y + h(t), t\big) - \Lambda\big( Q(t)y+h(t), t\big),\]
	and denote $\omega_v = \curl v$. It can be checked that $v \in L^2_\sigma(\mathcal{F}) \cap  H^s(\mathcal{F})$.  Therefore, we may apply Proposition \ref{Pro:Poisson-3} to obtain
	\begin{equation}
			\|\nabla v\|_{L^\infty(\mathcal{F})} \leqslant K(1+ \ln^+\|v\|_{H^3(\mathcal{F})})(1+\|v\|_{L^2(\mathcal{F})}+\|\omega_v\|_{L^\infty(\mathcal{F})}).
	\end{equation}
	
The orthogonality of $Q$ ensures that the above two estimates are preserved under the inverse transformation to the original variables. Furthermore, the Reynolds transport theorem implies the conservation of the $L^p$-norm ($p\in [1,\infty]$) of $\omega$ for all time.  The proof of the Proposition \ref{Pr:nabla-u-infty} is then concluded by applying Lemma \ref{Lm:boundary-corrector}.
\end{proof}

By using the boundary corrector $\Lambda$, we also can establish an estimate of $\|D^su\|_0$. Indeed, let $v$ and $\omega_v$ be defined as in the previous proposition, we infer from \eqref{Estimate:u-s} that
\begin{equation}\label{Est:d-s}
	\|D^s v\|_{L^2(\mathcal{F})} \leqslant K(\|\omega_v\|_{H^{s-1}(\mathcal{F})} + \|v\|_{L^2(\mathcal{F})}),
\end{equation} which combining with Lemma \ref{Lm:boundary-corrector} yields immediately the following result.
\begin{proposition}\label{Pro:high-u} Suppose that $(u, p, h, r) \in X_{s; \infty}$ is a solution of \eqref{Equ:euler-1}--\eqref{Equ:euler-4}. Then, for $t > 0$, we have
	\begin{equation}
		\|D^s u\|_0 \leqslant K\big(\|\omega\|_{s-1} + \|u\|_0 + |\dot{h}(t)| + |r(t)|\big).
	\end{equation} 
\end{proposition}

With the above results, the $H^s$ bound of $u$ can now be established.
\begin{proposition}\label{Pr:uniform}
	Let $T>0$ be fixed and $s \geqslant 3$. Suppose $(u, p, h, r) \in X_{s+1; \infty}$ is a solution to \eqref{Equ:euler-1}--\eqref{Equ:euler-4}. Then the energy bound \eqref{energy-inequality-thm} holds true.
\end{proposition}
\begin{proof}
Let us take curl of Eq.\eqref{Equ:euler-1} to get
	\begin{equation}
		\partial_t \omega + u \cdot \nabla \omega = 0.
	\end{equation}
	For every multi-index $|\alpha| < s$, we take the $\alpha$-$th$ order partial derivative on both sides of the above equation to obtain
	\begin{equation*}
		\partial_t \partial^\alpha \omega + u \cdot \nabla \partial^\alpha \omega + \sum_{|\beta| > 0} C_{\alpha, \beta}\partial^\beta u \cdot \nabla \partial^{\alpha-\beta} \omega = 0,
	\end{equation*}
	by  multiplying the above equation with $\partial^\alpha \omega$ and integrating over $\cup_{\tau\in[0, t]}\{\tau\}\times\mathcal{F}(\tau)$, we obtain
	\begin{equation*}\label{Equ:w-partial}
		\frac{1}{2}\int_0^t\int_{\mathcal{F}(\tau)}\big[\partial_\tau |\partial^\alpha \omega|^2 + u \cdot \nabla |\partial^\alpha \omega|^2\big] = - \sum_{|\beta| > 0} C_{\alpha, \beta}\int_0^t\int_{\mathcal{F}(\tau)}\partial^\beta u \cdot \nabla \partial^{\alpha-\beta} \omega \,\partial^\alpha\omega.
	\end{equation*}
	 Now, we apply the Reynolds transport theorem to get
	 \begin{equation} \label{Equ:weak-star}
	 	\frac{1}{2}\|\partial^\alpha \omega\|_{L^2(\mathcal{F}(t))}^2 = \frac{1}{2}\|\partial^\alpha\omega_0\|_{L^2(\mathcal{F})}- \sum_{|\beta| > 0} C_{\alpha, \beta}\int_0^t\int_{\mathcal{F}(\tau)}\partial^\beta u \cdot \nabla \partial^{\alpha-\beta} \omega\,\partial^\alpha\omega,
	 \end{equation}
	 where $\omega_0 = \curl u_0$. Summing up the above equations from $|\alpha|=0$ to $|\alpha|=s-1$ and using Proposition \ref{Pr:uniform-low} and \ref{Pro:high-u}, we deduce that 
	 \begin{equation*}
	 	\begin{split}
	 E_s(t) \leqslant K\Big(\int_0^t\|\nabla u\|_{L^\infty} \|u\|_{s}^2 \dif \tau + E_s(0)\Big).
	 	\end{split}
	 \end{equation*}
	 Noting that $\|u(\tau)\|_s^2 \leqslant E_s(\tau)$, we infer from Gr\"onwall's inequality that
	 \begin{equation}\label{Est:uniform-3}
	 	\begin{split}
	 	E_s(t) \leqslant KE_s(0)\exp{K\int_0^t\|\nabla u\|_{L^\infty}\dif \tau}.
	 \end{split}
	 \end{equation}
	 
	 Then, substituting \eqref{Est:uniform-3}, for $s=3$, into \eqref{Est:infty-u}, we arrive at
	 \begin{equation}
	 	\begin{split}
	 		&\|\nabla u\|_{L^\infty} \leqslant K\Big(1+\ln^+\sqrt{E_3(0)} + \int_0^t \|\nabla u\|_{L^\infty} \dif \tau\Big)\Big(1+\sqrt{E_1(0)}\Big),
	 	\end{split}
	 \end{equation}
	 where we have used the fact that $\sqrt{E_0(0)} + \|\omega_0\|_{L^\infty} \leqslant K\sqrt{E_1(0)}$.
it follows from	 Gr\"onwall's inequality  that
\begin{equation}
\|\nabla u\|_{L^\infty} \le K\bigl(1+\sqrt{E_1(0)}\bigr)
\bigl(1+\ln^+\sqrt{E_3(0)}\bigr)
\exp\Bigl(K\bigl(1+\sqrt{E_1(0)}\bigr)\,t\Bigr).
\end{equation}
	 
	 Finally, substituting the above inequality back into \eqref{Est:uniform-3}, we obtain \eqref{energy-inequality-thm} immediately, and this completes the proof of Proposition \ref{Pr:uniform}.
\end{proof}

\section{Proof of the main theorem}

Reference \cite{wang2012} established the local existence of a unique solution for the three-dimensional case under the assumption that the initial velocity field lies in $H^s$. The bidimensional case can be handled analogously, we thus present the following local existence result without proof.
\begin{proposition}\label{Pr:local-in-time-existence}
Fix $s \geqslant 3$. Assume that $u_0 \in H^s(\mathcal{F})$ satisfies \eqref{Equ:boundary-initial}. Then there exists some $T > 0$ such that the system \eqref{Equ:euler-1}--\eqref{Equ:euler-4} admits a unique solution $(u, p, h, w) \in \mathscr X_{s;T}$.
\end{proposition}

We next show that the solution can be extended globally.

\begin{proposition}\label{Pr:app-global-existence}
	With the assumptions of Proposition \ref{Pr:local-in-time-existence}, the system \eqref{Equ:euler-1}--\eqref{Equ:euler-4} admits a unique global solution $(u, p, h, r) \in \mathscr X_{s;\infty}$. Moreover, this solution satisfies the energy inequality \eqref{energy-inequality-thm}.
\end{proposition}
\begin{proof}
	To complete the proof, it suffices to verify \eqref{energy-inequality-thm}. It can be seen that the solution $(u, p, h, r)$ in Proposition \ref{Pr:uniform} was assumed to belong to $\mathscr X_{s+1; T}$ so that the integration in \eqref{Equ:w-partial} is well-defined. To address this issue, we need to construct a family of approximate solutions. 
	
	First, let us extend $u$ and $u_0$ to be vector fields on $\mathbb{R}^2$ such that their $H^s$ norms remain bounded. We denote by $\omega_0$ the curl of $u_0$, and let $\{\omega_0^\eps\}_{\eps >0} \subset H^s(\mathbb{R}^2)$ be a family of approximations for $\omega_0$ such that
	 \[\omega_0^\eps \to \omega_0 \text{ in } H^{s-1}(\mathbb{R}^2), \text{ as } \eps \to 0. \]
	
	Then, we infer from \cite{diperna1989ordinary,bahouri2011fourier} that there exists a unique solution \(\omega^\eps \in C([0, T]; \allowbreak  H^{s}(\mathbb{R}^2))\) to the following transport equation:
	\begin{equation}
		\begin{split}
			\partial_t \omega^\eps + u\cdot \nabla \omega^\eps = 0,\\
			\omega^\eps|_{t=0} =   \omega_0^\eps.
		\end{split}
	\end{equation}
	 For every multi-index $|\alpha| \leqslant s-1$, we deduce from the Reynolds transport theorem that
\begin{equation*} \label{Equ:rey}
	\begin{split}
	\frac{1}{2}\|\partial^\alpha &\omega^\eps\|_{L^2(\mathcal{F}(t))}^2 = \frac{1}{2}\|\partial^\alpha \omega_0^\eps\|_{L^2(\mathcal{F})}- \sum_{|\beta| > 0} C_{\alpha, \beta}\int_0^t\int_{\mathcal{F}(t)}\partial^\beta u \cdot \nabla \partial^{\alpha-\beta} \omega^\eps\partial^\alpha\omega^\eps.
	\end{split}
\end{equation*}

Finally, noting that $\omega^\varepsilon$ is uniformly bounded in $C\big([0,T]; H^{s-1}(\mathcal{F}(t))\big)$ and converges to $\omega$ in $C\big([0, T]; L^2(\mathcal{F}(t))\big)$, we apply the Banach-Alaoglu theorem and Fatou’s lemma to pass to the limit, obtaining \eqref{Equ:weak-star} with the equality replaced by an inequality. The energy bound \eqref{energy-inequality-thm} then follows from the remainder of the proof of Proposition \ref{Pr:uniform}.
\end{proof} 

\section{Acknowledgments}
	The work of X.~You is supported by the Jiangxi Provincial Natural Science Foundation of China (Grant No.~20242BAB25008) and the National Natural Science Foundation of China (Grant No.~12561038).



  \bibliographystyle{elsarticle-num} 
  \bibliography{my}





\end{document}